\documentclass[a4paper,12pt]{article}%
\usepackage{eurosym}
\usepackage{amsmath}
\usepackage{amssymb}
\usepackage{amsfonts}
\usepackage{amstext}
\usepackage{natbib}
\usepackage{verbatim}
\usepackage{tocloft}
\usepackage{geometry}
\usepackage{fancyhdr}
\usepackage{graphicx}%
\providecommand{\U}[1]{\protect\rule{.1in}{.1in}}
\newtheorem{theorem}{Theorem}

\newtheorem{definition}{Definition}

\newtheorem{proposition}{Proposition}
\newtheorem{remark}{Remark}

\newenvironment{proof}[1][Proof]{\noindent\textbf{#1.} }{\ \rule{0.5em}{0.5em}}
\begin{document}

\title{On the frequency variogram and on frequency domain methods for the analysis of
spatio-temporal data.}
\author{T. Subba Rao\\University of Manchester,U.K. and \\C. R. Rao AIMSCS, University of Hyderabad Campus, India
\and Gy. Terdik\\University of Debrecen, Hungary}
\date{}
\maketitle

\begin{abstract}
The covariance function and the variogram play very important roles in
modelling and in prediction of spatial and spatio-temporal data. The
assumption of second order stationarity, in space and time, is often made in
the analysis of spatial data and the spatio-temporal data. Several times the
assumption of stationarity is considered to be very restrictive, and
therefore, a weaker assumption that the data is Intrinsically stationary both
in space and time is often made and used, mainly by the geo-statisticians and
other environmental scientists. In this paper we consider the data to be
\ intrinsically stationary. Because of the inclusion of time dimension,the
estimation and derivation of the sampling properties of various estimators
related to spatio-temporal data become complicated. In this paper our object
is to present an alternative way, based on Frequency Domain methods for
modelling the data. Here we consider Discrete Fourier Transforms (DFT) defined
for the (Intrinsic) time series data observed at several locations as our
data, and then consider the estimation of the parameters of spatio-temporal
covariance function, estimation of Frequency Variogram, tests of independence
etc. We use the well known property that the Discrete Fourier Transforms of
stationary time series evaluated at distinct Fourier Frequencies are
asymptotically independent and distributed as complex normal in deriving many
results considered in this paper. Our object here is to emphasize the
usefulness of the Discrete Fourier transforms in the analysis of
spatio-temporal data. Under the intrinsic stationarity condition we consider
the estimation, discuss the sampling properties of the Frequency Variogram
(FV) introduced in an earlier paper by \cite{Subba1} which was proposed as an
alternative to the classical space, time variogram. We show that the FV
introduced is a frequency decomposition of the space-time variogram, and can
be computed using the Fast Fourier Transform algorithms. Assuming that the
DFT's of the intrinsically stationary processes satisfy a Laplacian type of
model, an analytic expression for the space-time spectral density function is
derived for the intrinsic processes and also an expression for the Frequency
Variogram in terms of the spectral density function is also derived. The
estimation of the parameters of the spectrum is also considered. A statistical
test for spatial independence of spatio-temporal data is also briefly
mentioned, and is based on the test proposed earlier by \cite{Wahb} for
testing independence in multivariate stationary (temporally) time
series.\smallskip\newline\textit{Keywords.} Intrinsic stationarity,
spatio-temporal random Processes, Frequency Variogram, Laplacian Model, Test
for spatial Independence.\smallskip\newline\textbf{Dedication.} Professor M.
B. Priestley has made many significant contributions to the nonparametric
estimation of stationary and nonstationary spectral density functions. He was
one of the strong believers of the use of Fourier Transforms, Frequency domain
methods in the analysis of time series. This paper is based on the Fourier
Transforms and their possible application to spatio-temporal data and is
written bearing in mind Professor Priestley's many important contributions in
this area. We dedicate this paper to him.

\end{abstract}

\section{Introduction and Summary}

\label{sect:intro} Spatio temporal data arises in many areas such as
agriculture, geology, environmental sciences, finance, etc. Since the data
comes from these areas are functions of both time and space, any statistical
method developed must take into account both spatial dependence, temporal
dependence and any interaction between these two. In the case of spatial data,
the second order spatial dependence is measured by the second order covariance
function and if the spatial process is second order stationary, then the
second order covariance is a function of spatial lag only. In the case of
spatio-temporal data the dependence is measured by space-time covariance
function and if the process is spatially and temporally stationary, then the
covariance function is a function of the spatial lag and temporal lag. These
functions are usually estimated under the assumption that the random process
is spatially and temporally stationary.

An alternative second order dependence measure is the variogram defined for
both spatial processes and spatio-temporal processes. This function is well
defined under the weaker assumption \ of intrinsic stationarity and in view of
this it is widely used in geo-statistics. Its use is strongly advocated by
\cite{Cressie}, \cite{gringarten2001teacher} and \cite{Sherman} and many others.

If the process is second order stationary, then there is a one to one
correspondence between the variogram and the covariance function. The
estimation of the spatial covariance, spatial variogram \ and their asymptotic
sampling properties \ have been considered by several authors \cite{Cressie},
\cite{Yun}, \cite{Stein}, \cite{Gneit}, \cite{Huang},
\cite{gringarten2001teacher}, \cite{ma2005spatio}. The literature on the
estimation of space-time covariance function and the space-time variogram is
not very extensive in the case of spatio-temporal random processes. The
inclusion of temporal dimension complicates the estimation. The estimation and
the sampling properties of the spatio- temporal covariance function have been
briefly considered by \cite{Li}, \cite{Hucre}, \cite{Stein2}.

In this paper our objective is to consider the Discrete Fourier Transforms
(DFT) of the time series evaluated at Fourier frequencies as our data. If the
observed time series data is equally spaced, one can use the Fast Fourier
Transform (FFT) algorithm to compute the DFTs. Using the DFT's we model the
data. \cite{Subba1} and \cite{2013arXivSRT_GyT} \ use the recently defined '
Frequency variogram ' for the estimation of the parameters of spatio-temporal
covariance function of the process assuming that the DFT's satisfy a Complex
Stochastic Partial Differential Equation(CSPDE).

We show the spatio temporal variogram and the frequency variogram defined
earlier are related. The non-parametric estimation of the frequency variogram
is considered. Its sampling properties are discussed. Investigation of the
sampling properties of the sample Frequency Variogram is much easier compared
to the space-time variogram estimate. We believe that many interesting
problems associated with spatio-temporal random processes can be solved using
the frequency domain methods. We consider here some of these problems.

We \ now summarize the contents of the paper. In Section 2, the space time
covariance function and space time variogram are introduced, and their
estimation, under the assumption of stationarity, is discussed in Section 3.
The properties of Discrete Fourier Transforms of stationary spatial processes,
spectral representation of the processes are considered in Section 4. The
Frequency Variogram and its relation to the classical spatio-temporal
variogram, and the non-parametric estimation of the Frequency variogram are
considered in Sections 5 and 6 and these are considered under the assumption
of Intrinsic stationarity of the process. Assuming that the process is
intrinsically stationary, and the intrinsic process satisfies a Laplacian
model, an analytic expression for the spectral density of the intrinsic
process is obtained in Section 7. The estimation of the parameters of the
spectral density function of the Intrinsic process obtained in section 7, is
considered in Section 8. The frequency variogram and its relation to the
spectral density function is also considered in Section 8. A test for spatial
independence, based on the properties of Complex Wishart distribution, is
described in Section 9 and the test is based on the test for independence by
\cite{Wahb}.

\section{Space-time Covariance function and the \newline Space-Time
Variogram.}

Let $\{Y_{t}(\mathbf{s}),\mathbf{s}\in\mathbb{R}^{d},t\in\mathbb{Z}\}$ denote
the spatio-temporal random process. Two assumptions are often made which are
important for modeling and prediction. They are that the process is second
order stationary in space and time and also that the process is isotropic in
space. The assumption of stationarity can be sometimes unrealistic. In view of
this, another weaker assumption that is often made is that the process is
intrinsically stationary. We note that if the process is second order
stationary, then it implies that the process is intrinsically stationary. But
the converse is not true. We say the process $\{Y_{t}(\mathbf{s})\}$ is
spatially, temporally second order stationary if, for all $t\in Z,\mathbf{s\in
R}^{d},$%
\begin{align*}
E[Y_{t}(\mathbf{s})]  &  =\mu,\\
Var[Y_{t}(\mathbf{s})]  &  =c(0,0)={\sigma_{y}}^{2}<\infty,\\
Cov\left[  Y_{t}(\mathbf{s}),Y_{t+u}(\mathbf{s}+\mathbf{h})\right]   &
=c(\mathbf{h},u),\mathbf{h}\in\mathbb{R}^{d},u\in\mathbb{Z.}%
\end{align*}

We note $c(\mathbf{h},0)$ and $c(0,u)$ correspond to the purely spatial,
purely temporal covariances respectively. Without loss of any generality we
assume that $\mu=0$.

The random process is said to be isotropic if%
\[
c(\mathbf{h},u)=c\left(  \Vert\mathbf{h}\Vert;u\right)  ,\mathbf{h}%
\in\mathbb{R}^{d},u\in\mathbb{Z},
\]
where $\Vert\mathbf{h}\Vert$ is the Euclidean distance. The process is said to
be fully symmetric if $c(\mathbf{h},u)=c(-\mathbf{h},u)=c(\mathbf{h}%
,-u)=c(-\mathbf{h},-u)$ (see \cite{Gneit2}). The process $\{Y_{t}%
(\mathbf{s})\}$ is intrinsically spatially, temporarily stationary if the
incremental process, for $u\in Z,\mathbf{h\in R}^{d}$, $Y_{t}(\mathbf{s}%
)-Y_{t+u}(\mathbf{s}+\mathbf{h})$ satisfies the following (see \cite{Cressie2}%
,p.315)
\begin{align*}
E\left(  (Y_{t}(\mathbf{s})-Y_{t+u}(\mathbf{s}+\mathbf{h})\right)   &  =0,\\
Var\left[  (Y_{t}(\mathbf{s})-Y_{t+u}(\mathbf{s}+\mathbf{h})\right]   &
=\gamma(\mathbf{h},u)<\infty..
\end{align*}
\newline If $\{Y_{t}(\mathbf{s})\}$ is isotropic, then
\[
\gamma(\mathbf{h},u)=\gamma(\left\Vert \mathbf{h}\right\Vert ,u),
\]
where $\gamma(\mathbf{h},u)$ is also known as the structure function(
\cite{Yag}.)

The spatio-temporal variogram is defined as
\[
\gamma(\mathbf{h},u)=2\tilde{\gamma}(\mathbf{h},u)=Var\left[  (Y_{t}%
(\mathbf{s})-Y_{t+u}(\mathbf{s}+\mathbf{h})\right]  ,
\]
and $\tilde{\gamma}(u,\mathbf{h})$ is defined as the semi spatio-temporal
variogram. We note that one can define the variogram under the weaker
assumption of intrinsic stationarity. In other words we do not need the
assumption of stationarity of the original processes. This phenomenon of
differencing\ \ in space to achieve stationarity is similar to what we have in
the case of random processes with stationary increments in time, for instance,
the Brownian motion.

Suppose the process$\{Y_{t}(\mathbf{s})\}$ is spatially and temporally
stationary, then we can show
\begin{align*}
\gamma(\mathbf{h},u)  &  =2\left[  Var(Y_{t}(\mathbf{s}))-Cov\left(
Y_{t}(\mathbf{s}),Y_{t+u}(\mathbf{s}+\mathbf{h})\right)  \right] \\
&  =2\left[  c(0,0)-c(\mathbf{h},u)\right]  =2\tilde{\gamma}(\mathbf{h},u),
\end{align*}
and we note that there is a one to one correspondence between $\gamma
(\mathbf{h},u)$ and $c(\mathbf{h},u)$ in the case of stationary processes. One
can show that the covariance function $c(\mathbf{h},u)$ is positive
semi-definite and $\gamma(\mathbf{h},u)$ is conditionally negative definite.

\section{Estimation of $c(\mathbf{h},u)$ and $\gamma\left(  \mathbf{h}%
,u\right)  .$}

Let $\left\{  Y_{t}(\mathbf{s}_{i})\text{; }i=1,2,....,m\text{; }%
t=1,2,....,n\right\}  $ be a sample from the zero mean, stationary spatio-
temporal random process ${Y_{t}(\mathbf{s})}$. We define the estimates of
$c(\mathbf{h},u)$ and $\gamma(\mathbf{h},u)$ as follows. (see \cite{Sherman}
\ for details). Let
\[
\hat{c}(\mathbf{h},u)=\frac{1}{|N(\mathbf{h},u)|}\sum_{N(\mathbf{h}%
,u)}[Y_{t_{i}}(\mathbf{s}_{i})-\overline{Y}(\mathbf{s}_{i})][Y_{t_{j}%
}(\mathbf{s}_{j})-\overline{Y}(\mathbf{s}_{j})],
\]
where
\[
\overline{Y}(\mathbf{s}_{i})=\frac{1}{n}\sum_{t=1}^{n}Y_{t}(\mathbf{s}_{i}),
\]
and
\[
\hat{\gamma}(\mathbf{h},u)=\frac{1}{|N(\mathbf{h},u)|}\sum_{N(\mathbf{h}%
,u)}{\left[  Y_{t_{i}}(\mathbf{s}_{i})-[Y_{t_{j}}(\mathbf{s}_{j})\right]
}^{2},
\]
where $N(\mathbf{h},u)=\left\{  (\mathbf{s}_{i},t_{i}),(\mathbf{s}_{j}%
,t_{j});\mathbf{s_{i}}-\mathbf{\mathbf{s}_{j}}=\mathbf{h}\quad\text{and}\quad
t_{i}-t_{j}=u\right\}  $. The estimator $\hat{\gamma}(\mathbf{h},u)$ is widely
known as Matheron estimator. In this paper \ we are assuming that the time
series data observed at all $m$ locations are equally spaced and also there
are no missing values. It is interesting to investigate the properties of the
estimators proposed here when these assumptions do not hold.

Under certain conditions, \cite{Li} have shown that the sample spatio-temporal
covariance function defined above is asymptotically normal.

Based on $\hat{\gamma}(\mathbf{h},u)$, \cite{Cressie} and \cite{Huang} have
proposed a weighted least squares criterion for estimating the parameters of
the theoretical variogram $\gamma(\mathbf{h},u|\theta)$, and \cite{Gneit2}
proposed a similar criterion for estimating the parameters based on the space-
time covariance function $\hat{c}(\mathbf{h},u)$. \cite{Subba1} have proposed
a frequency domain method for the estimation of the parameters which is robust
against departures from Gaussianity and also computationally efficient. The
method of estimation proposed by \cite{Subba1} is similar to Whittle
likelihood approach and \ it is based on the frequency variogram, and the
proposed criterion which is easy to compute and is based on Discrete Fourier
Transforms. In the following section we define the Frequency Variogram (FV)
and derive the sampling properties of the estimator.

\section{Discrete Fourier transforms and the spectral representation of the
process $\left\{  Y_{t}(\mathbf{s})\right\}  .$}

We follow the notation introduced in the paper of \cite{2013arXivSRT_GyT}.
Here we briefly highlight and summarize the results we need for our present
purposes and for further details we refer to \cite{2013arXivSRT_GyT} and the
books and papers cited in those papers.

We assume the random process $\left\{  Y_{t}(\mathbf{s})\right\}  $ is second
order spatially and temporally stationary. Therefore, the process has the
spectral representation given by
\[
Y_{t}(\mathbf{s})=\int_{R^{d}}\int_{-\pi}^{\pi}e^{i(\mathbf{s}\cdot
\boldsymbol{\lambda}+t\omega)}dZ_{y}(\boldsymbol{\lambda},\omega),
\]
where $\mathbf{s}\cdot\boldsymbol{\lambda}=\sum_{i=1}^{d}s_{i}\lambda_{i}$ and
$\int_{R^{d}},$ represents a d-fold multiple integral, and $Z_{y}%
(\boldsymbol{\lambda},\omega)$ is a zero mean complex valued random process
with orthogonal increments and
\begin{align*}
E[dZ_{y}(\boldsymbol{\lambda},\omega)]  &  =0,\\
E{|dZ_{y}(\boldsymbol{\lambda},\omega)|}^{2}  &  =dF_{y}(\boldsymbol{\lambda
},\omega),
\end{align*}
where $dF_{y}(\boldsymbol{\lambda},\omega)$ is a spectral measure. If we
assume further that $dF_{y}(\boldsymbol{\lambda},\omega)$ is absolutely
continuous with respect to Lebesgue measure according to the arguments
$\boldsymbol{\lambda}$ and $\omega$, then $dF_{y}(\boldsymbol{\lambda}%
,\omega)=f_{y}(\boldsymbol{\lambda},\omega)d\boldsymbol{\lambda}d\omega$,
where $d\boldsymbol{\lambda}=\prod_{i=1}^{d}d{\lambda}_{i}$. Here
$f_{y}(\boldsymbol{\lambda},\omega)$ is a strictly positive, real valued
function and is defined as the spatio- temporal spectrum of the random process
$\left\{  Y_{t}(\mathbf{s})\right\}  $, and $-\infty<{\lambda}_{1},{\lambda
}_{2},...,{\lambda}_{d}<\infty,-\pi\leq\omega\leq\pi$. In view of the
orthogonality of the function $Z_{y}(\boldsymbol{\lambda},\omega)$, it can be
shown that
\begin{equation}
c(\mathbf{h},u)=\int_{\mathbf{R}^{d}}\int_{-\pi}^{\pi}e^{i(\mathbf{h}%
\cdot\boldsymbol{\lambda}+u\omega)}f_{y}(\boldsymbol{\lambda},\omega)d\omega
d\boldsymbol{\lambda,} \label{eq4.1}%
\end{equation}
and by inversion we get
\begin{equation}
f_{y}(\boldsymbol{\lambda},\omega)=\frac{1}{{(2\pi)}^{d+1}}\sum_{u}%
\int_{-\infty}^{\infty}e^{-i(\mathbf{h}\cdot\boldsymbol{\lambda}+u\omega
)}c(\mathbf{h},u)d\mathbf{h.} \label{Spectrum}%
\end{equation}
From (\ref{eq4.1}), we have
\begin{equation}
c(\mathbf{0},u)=\int_{-\pi}^{\pi}e^{iu\omega}g_{0}(\omega)d\omega,
\label{Covariance}%
\end{equation}
where $g_{0}(\omega)=\int_{-\infty}^{\infty}f_{y}(\boldsymbol{\lambda}%
,\omega)d\boldsymbol{\lambda}$ is the second order temporal spectral density
function of the process $\left\{  Y_{t}(\mathbf{s})\right\}  $, and in view of
our assumption that the process is spatially, temporally stationary
$g_{0}(\omega)$ is same for all the locations $\mathbf{s}$. We note
$c(\mathbf{h},u)=c(-\mathbf{h},-u)$ and $f_{y}(\boldsymbol{\lambda}%
,\omega)=f_{y}(-\boldsymbol{\lambda},-\omega)$, and $f_{y}(\boldsymbol{\lambda
},\omega)>0$ for all $\boldsymbol{\lambda}$ and $\omega$.

Here $\boldsymbol{\lambda}$ is the spatial frequency associated with the
spatial coordinates $\mathbf{s}_{i}$ and is usually called the wave number and
$\omega$ is the temporal frequency associated with time.

Let $\left\{  Y_{t}(\mathbf{s_{i}})\right\}  ;i=1,2,...,m;t=1,2,...n$ be a
sample from the zero mean, stationary spatio-temporal random process $\left\{
Y_{t}(\mathbf{s})\right\}  $. Consider the time series data at the location
$\mathbf{s}_{i}$ and define the Discrete Fourier transform (DFT)
\begin{equation}
J_{\mathbf{s}_{i}}^{y}(\omega_{k})=\frac{1}{\sqrt{2\pi n}}\sum_{t=1}^{n}%
Y_{t}(\mathbf{s}_{i})e^{-it\omega_{k}};\text{ \ \ }(i=1,2,...,...,m)
\label{eq4.3}%
\end{equation}
where $\omega_{k}=\frac{2\pi k}{n},k=0,1,2,...,\left[  \frac{n}{2}\right]  $.
We note that the Discrete Fourier transforms can be evaluated using the Fast
Fourier Transform (FFT) algorithm(FFT), and the number of operations required
to calculate FFT from a time series of length n, is of the order $n(\ln n)$.
By inversion, we obtain from (\ref{eq4.3})
\begin{equation}
Y_{t}(\mathbf{s})=\sqrt{\frac{n}{2\pi}}\int_{-\pi}^{\pi}e^{it\omega
}J_{\mathbf{s}}^{y}(\omega)d\omega. \label{A1}%
\end{equation}
The above representation shows that the $\left\{  Y_{t}(\mathbf{s})\right\}  $
can be decomposed into sine and cosine terms and the complex valued random
variable DFT, $J_{s}^{y}(\omega)$ can be considered as the amplitude
corresponding to these sine and cosine basis functions.

We will briefly summarize some well known results associated with DFT's (see
Appendix) which will be required later. \ For details regarding properties of
the Discrete Fourier Transforms for stationary processes, we refer to the
books \ of \cite{Brilli} and \cite{giraitis2012large}. It is well known that
\ under some structural assumptions (see \cite{giraitis2012large}) the
discrete Fourier transforms $\left\{  J_{\mathbf{s}}^{y}(\omega_{k})\right\}
$ evaluated at discrete Fourier frequencies ${\omega_{k}}$ are asymptotically
uncorrelated, and is distributed as complex normal (see for details
\cite{Brilli} \ and \cite{giraitis2012large}.

For example, for large $n$, and for a specific $\omega_{k}$ and for a specific
$\mathbf{s}$, $\left\{  J_{\mathbf{s}}^{y}(\omega_{k})\right\}  $ is
approximately distributed as complex normal with mean zero and variance
$\left\{  g_{\mathbf{s}}(\omega_{k})\right\}  $ which is the second order
temporal spectrum of the process at the location $\mathbf{s}$. In view of the
spatial stationarity assumption, $g_{\mathbf{s}}(\omega_{k})$ is same for all
locations, and we denote this common temporal spectrum by $g_{0}(\omega_{k})$.

Let $I_{\mathbf{s}}^{y}(\omega_{k})={|J_{\mathbf{s}}^{y}(\omega_{k})|}^{2}$ be
the periodogram, and let $I_{\mathbf{s}_{i},\mathbf{s}_{j}}^{y}(\omega
_{k})=J_{\mathbf{s}_{i}}^{y}(\omega_{k}){J_{\mathbf{s}_{j}}^{y\ast}}%
(\omega_{k})$ be the cross periodogram between the two time series $\left\{
Y_{t}(\mathbf{s}_{i})\right\}  $ and $\left\{  Y_{t}(\mathbf{s}_{j})\right\}
.$ In the appendix we summarize some properties of the periodograms (see also
\cite{2013arXivSRT_GyT}). In the following section, we define the Frequency
Variogram and consider its estimation and also discuss the asymptotic sampling
properties of the estimator proposed.

\section{Frequency Variogram (FV), Properties and its estimation.}

As stated earlier, variogram is used as an alternative measure of second order
dependence. It can be defined under weaker conditions and as such it is widely
used. Though the statistical properties of the sample variogram are well
studied in the case of spatial processes, the estimation and the asymptotic
properties of various estimators defined for spatio-temporal processes, such
as $\hat{\gamma}(\mathbf{h},u)$ defined earlier are not well investigated and
this could be due to the inclusion of the time dimension in the processes. To
circumvent such problems, \cite{Subba1} have considered frequency domain
approach for the statistical analysis, model construction and estimation.

The authors \cite{Subba1} have introduced frequency variogram as an
alternative to spatio-temporal variogram defined earlier and was found to be
very useful in the estimation of parameters of spatio-temporal spectrum. As no
inversion of high dimensional matrices are required in the estimation
suggested, the computation of the minimizing criterion is easy. In this paper
we consider further properties of the Frequency Variogram and also discuss its
nonparametric estimation. We use the FV as a tool for estimating the
parameters of the spatio-temporal spectrum of the intrinsic processes.

Let $\left\{  J_{\mathbf{s}}^{y}(\omega_{k})\right\}  $ be the DFT evaluated
at the Fourier frequency $\omega_{k}=\frac{2\pi k}{n};$ $k=0,1,2,...,\left[
\frac{n}{2}\right]  $ calculated using the time series data $\left\{
Y_{t}(\mathbf{s})\right\}  $.

The frequency variogram is defined, for a fixed spatial lag $\mathbf{h}$ and
at the location $\mathbf{s}$, as follows.

Let
\[
X_{t}^{\mathbf{h}}(\mathbf{s})=Y_{t}(\mathbf{s})-Y_{t}(\mathbf{s+h}%
),t=1,2,...n.
\]
We have
\[
E[X_{t}^{\mathbf{h}}(\mathbf{s})]=0,Var[X_{t}^{\mathbf{h}}(\mathbf{s}%
)]=\gamma(\mathbf{h,}0\mathbf{).}%
\]
Define the DFT of the time series \ $\{X_{t}^{\mathbf{h}}(\mathbf{s})\}$ by%
\[
\ J_{\mathbf{s,s+h}}^{x}(\omega)=\frac{1}{\sqrt{2\pi n}}\sum_{t=1}^{n}%
X_{t}^{\mathbf{h}}(\mathbf{s)}e^{-it\omega}=J_{\mathbf{s}}^{y}(\omega
)-{\normalsize J}_{\mathbf{s+h}}^{y}(\omega),
\]
and the periodogram by%
\[
{\normalsize I}_{\mathbf{s,s+h}}^{x}(\omega)\ \ =|J_{\mathbf{s,s+h}}%
^{x}(\omega)|^{2}.
\]
\ \ \ \ \ \ 

\begin{definition}%
\begin{align*}
G_{\mathbf{s,s+h}}^{x}(\omega)  &  =2\widetilde{G}_{\mathbf{s,s+h}}^{x}%
(\omega)\\
&  =E|J_{\mathbf{s}}^{y}(\omega)-J_{\mathbf{s}+\mathbf{h}}^{y}(\omega)|^{2}\\
&  =E[I_{\mathbf{s,s+h}}^{x}(\omega)],
\end{align*}
for all $|\omega|\leq\pi$. \cite{Subba1} \ defined $G_{\mathbf{s,s+h}}%
^{x}(\omega)$ as the Frequency Variogram.
\end{definition}

We note $\ $\ that $J_{\mathbf{s,s+h}}^{x}(\omega)$ is the DFT of the
incremental random process $\left\{  X_{t}^{\mathbf{h}}(\mathbf{s})\right\}
$. If the incremental process defined is spatially intrinsically stationary,
and also temporally stationary, then the discrete Fourier transforms $\left\{
J_{\mathbf{s,s+h}}^{x}(\omega_{k})\right\}  $ are asymptotically uncorrelated,
and distributed as Complex Gaussian (\cite{Brilli} and
\cite{giraitis2012large} ). These functions are well defined and no
assumptions of spatial, temporal stationarity of the process $\left\{
Y_{t}(\mathbf{s})\right\}  $ is required. The FV $G_{\mathbf{s,s+h}}%
^{x}(\omega)$ can be used as a measure of dissimilarity between the two random
process $\left\{  Y_{t}(\mathbf{s})\right\}  $ and $\left\{  Y_{t}%
(\mathbf{s+h})\right\}  $ at the frequency $\omega$. As one would expect this
measure to increase as the spatial lag $\left\Vert \mathbf{h}\right\Vert $
increases and tends to zero as $||\mathbf{h||\rightarrow0}$. Some further
comments on FV are in order.

\begin{remark}
The FV given by $G_{\mathbf{s,s+h}}^{x}(\omega)$ is well defined and defined
under the weaker condition of Intrinsic stationarity.
\end{remark}

\begin{remark}
If the intrinsic process $\left\{  X_{t}^{\mathbf{h}}(\mathbf{s})\right\}  $
is spatially and temporally stationary, its second order periodogram
${\normalsize I}_{\mathbf{s,s+h}}^{x}(\omega)$ is asymptotically an unbiased
estimator of the temporal spectrum of the intrinsic process$\left\{
X_{t}^{\mathbf{h}}(\mathbf{s})\right\}  .$ In view of the assumption of the
spatial stationarity of the intrinsic process, the second order spectrum does
not depend on the location $\mathbf{s.}$ Therefore estimating the FV is same
as estimating the second order spectral density function of the intrinsic
process $\left\{  X_{t}^{\mathbf{h}}(\mathbf{s})\right\}  .$ This estimation
is considered in section 6.
\end{remark}

In the following we show the relationship between the spatio temporal
variogram $\gamma(\mathbf{h},u)$ and the FV.

\begin{proposition}
Let%
\[
G_{\mathbf{s,s+h}}^{x}(\omega)=E{|J_{\mathbf{s,s+h}}^{x}(\omega)|}^{2},
\]
then%
\begin{equation}
\int_{-\pi}^{\pi}G_{\mathbf{s,s+h}}^{x}(\omega)d\omega=\gamma(\mathbf{h},0).
\label{pars}%
\end{equation}

\end{proposition}

\begin{proof}
An application of Parseval's theorem gives the above result.
\end{proof}

In the derivation of the above we used the assumption that the incremental
process $\left\{  X_{t}^{\mathbf{h}}(\mathbf{s})\right\}  $ is stationary
temporally and spatially even though the original process $\left\{
Y_{t}(\mathbf{s})\right\}  $ may not be spatially, temporally stationary.

The above result (\ref{pars}) shows that the FV, $G_{\mathbf{s,s+h}}%
^{x}(\omega)$ is the frequency decomposition of the classical spatio temporal
variogram $\gamma(\mathbf{h},u)$ when $u=0$, similar to the frequency
decomposition we have for the power (variance) of the stationary random
process in terms of the power spectral density function. Since $\gamma
(\mathbf{h},u)$ is a measure of dissimilarity between two spatial processes
separated by lag $\mathbf{h}$, $G_{\mathbf{s,s+h}}^{x}(\omega)$ is also a
measure of dissimilarity of the two process at the frequency $\omega$. By
plotting this function as a function of $\omega$, one can observe in which
frequency band there is a large amount of lack of similarity. This information
could be useful in prediction where one can predict a time series using the
time series data from other neighborhood locations.

\begin{proposition}
.Let $\left\{  {Y}_{t}(\mathbf{s})\right\}  $ be a zero-mean second order
stationary process in space and time and let $\left\{  J_{\mathbf{s}_{i}}%
^{y}(\omega)\right\}  (i=1,2,...,m)$ be the DFT 's of $\left\{  {Y}%
_{t}(\mathbf{s}_{i}),i=1,2,...m\right\}  $. Let $G_{\mathbf{s}_{i}%
,\mathbf{s}_{j}}^{x}(\omega)$ be the frequency variogram. Then

\begin{enumerate}
\item The covariance function $g_{\mathbf{s}_{i},\mathbf{s}_{j}}^{y}%
(\omega)=cov(J_{\mathbf{s}_{i}}^{y}(\omega),J_{\mathbf{s}_{j}}^{x}(\omega))$
is a positive semi-definite function.

\item The FV $G_{\mathbf{s}_{i},\mathbf{s}_{j}}^{x}(\omega)$ is conditionally
negative definite.
\end{enumerate}
\end{proposition}

\begin{proof}
Consider the sum $S_{1}(\omega)=\sum_{i=1}^{m}a_{i}J_{\mathbf{s}_{i}}%
^{y}(\omega)$ where $\left\{  a_{i}\right\}  $ can be complex. Then
\[
VarS_{1}(\omega)=\sum\sum a_{i}{a_{j}}^{\ast}Cov(J_{\mathbf{s}_{i}}^{y}%
(\omega),J_{\mathbf{s}_{j}}^{y}(\omega))\geq0.
\]
Hence the result (1).\newline To prove the second result, assume $\sum
a_{i}=0$. Then we can show that%
\begin{align*}
\sum\sum a_{i}{a_{j}}^{\ast}G_{\mathbf{s}_{i},\mathbf{s}_{j}}^{x}(\omega)  &
=\sum\sum a_{i}{a_{j}}^{\ast}E{|J_{\mathbf{s}_{i}}^{y}(\omega)-J_{\mathbf{s}%
_{j}}^{y}(\omega)|}^{2}\\
&  =-2\text{ }Var[\sum_{i=1}^{m}a_{i}J_{\mathbf{s}_{i}}^{y}(\omega)]\leq0.
\end{align*}
In the above derivation we used the fact that $\sum a_{i}=0$ and also the
second order spectral density function does not depend on the location
$\mathbf{s}_{i}$ because of stationarity assumption. Hence the result (2) of
the proposition.
\end{proof}

\subsection{Frequency Variogram and Nugget Effect:}

For illustration purposes we consider the case $d=2$. Suppose instead of
observing the process $\left\{  {Y}_{t}(\mathbf{s}),\mathbf{s}\in
\mathbb{R}^{2},t\in\mathbb{Z}\right\}  $, we observe a corrupted random
process $\left\{  \widetilde{Y}_{t}(\mathbf{s}),\mathbf{s}\in\mathbb{R}%
^{2},t\in\mathbb{Z}\right\}  $, where for each $\mathbf{s}$ and $t$,%
\[
\widetilde{Y}_{t}(\mathbf{s})=Y_{t}(\mathbf{s})+\eta_{t}(\mathbf{s}),
\]
and $\left\{  {Y}_{t}(\mathbf{s})\right\}  $ and $\left\{  \eta_{t}%
(\mathbf{s})\right\}  $ are zero mean spatially, temporally stationary
processes and $\left\{  {Y}_{t}(\mathbf{s})\right\}  $ and $\left\{  \eta
_{t}(\mathbf{s})\right\}  $ are independent for all $t$ and $\mathbf{s}$, it
is defined as a generalized process. Further, we assume that $\left\{
\eta_{t}(\mathbf{s})\right\}  $ is a white noise process in space and time
with the second order space-time spectrum $g_{\eta}(\boldsymbol{\lambda
},\omega)=\frac{{\sigma_{\eta}}^{2}}{{(2\pi)}^{3}}$ for all
$\boldsymbol{\lambda}$ and $\omega$. Define the DFT of the incremental random
process of $\{\widetilde{Y}_{t}(\mathbf{s})\},$
\[
\ (\widetilde{Y}_{t}(\mathbf{s})-\widetilde{Y}_{t}(\mathbf{s+h}))=(Y_{t}%
(\mathbf{s})-Y_{t}(\mathbf{s+h}))+(\eta_{t}(\mathbf{s})-\eta_{t}%
(\mathbf{s+h}),
\]
then we have
\[
\widetilde{J}_{\mathbf{s,s+h}}(\omega)={J_{\mathbf{s,s+h}}^{x}}(\omega
)+{J_{\mathbf{s,s+h}}^{\eta}}(\omega),\quad|\omega|\leq\pi,
\]
where%
\begin{align*}
\widetilde{J}_{\mathbf{s,s+h}}(\omega)  &  =\frac{1}{\sqrt{2\pi n}}%
\sum(\widetilde{Y}_{t}(\mathbf{s})-\widetilde{Y}_{t}(\mathbf{s}+\mathbf{h}%
))e^{-i\omega t},\\
J_{\mathbf{s,s+h}}^{x}(\omega)  &  =\frac{1}{\sqrt{2\pi n}}\sum({Y}%
_{t}(\mathbf{s})-{Y}_{t}(\mathbf{s}+\mathbf{h}))e^{-i\omega t},\\
{J_{\mathbf{s,s+h}}^{\eta}}(\omega)  &  =\frac{1}{\sqrt{2\pi n}}\sum(\eta
_{t}(\mathbf{s})-\eta_{t}(\mathbf{s}+\mathbf{h}))e^{-i\omega t}.
\end{align*}
Define the FV for the process $\left\{  \widetilde{Y}_{t}(\mathbf{s})\right\}
$,%
\begin{align}
{\widetilde{G}}_{\mathbf{s,s+h}}(\omega)  &  =E{|}\widetilde{J}%
{{_{\mathbf{s,s+h}}}(\omega)|}^{2}\nonumber\\
&  =E{|{{J{{^{x}}_{\mathbf{s,s+h}}}}}(\omega)|}^{2}+E{|{{J^{\eta}%
}_{\mathbf{s,s+h}}}(\omega)|}^{2}\nonumber\\
\text{ }  &  ={G}_{\mathbf{s,s+h}}^{x}(\omega)+\frac{2\sigma_{\eta}^{2}%
}{{(2\pi)}^{3}}. \label{eq5.1}%
\end{align}
.

The above result follows because of our assumption that the random process
$\left\{  \eta_{t}(\mathbf{s})\right\}  $ is a white noise. \ From
(\ref{eq5.1}), we observe that as $\left\Vert \mathbf{h}\right\Vert
\rightarrow0$, $G_{\mathbf{s,s+h}}^{x}(\omega)\rightarrow0$ for all $\omega$
and, therefore, $\widetilde{G}_{\mathbf{s,s+h}}(\omega)\rightarrow\frac
{\sigma_{\eta}^{2}}{{(2\pi)}^{3}}$ as $\left\Vert \mathbf{h}\right\Vert
\rightarrow0$.

If we plot $\int G_{\mathbf{s,s+h}}(\omega)d\omega$ as a function of
$\left\Vert \mathbf{h}\right\Vert $ and if we observe a jump near the origin
$\left\Vert \mathbf{h}\right\Vert $ =$0$, this could be due to the presence of
\ white noise in the process. In other words, the observations are corrupted
by white noise. This effect is usually called the "Nugget effect" in
geo-mining literature. In the following section we consider the estimation of
$G_{\mathbf{s,s+h}}(\omega)$ when the observations are not corrupted. In
practice one uses the Fast Fourier Transform algorithm for computing the DFT's
when the time series data is equally spaced.

We may point out that other types of Nugget effects are feasible, for example
one could have a process which is temporally correlated, but spatially
uncorrelated. Such processes \ were discussed by \cite{stein2005statistical}.

\section{Estimation of the Frequency variogram under the Intrinsic
stationarity.}

Let $\left\{  Y_{t}(\mathbf{s}_{i});i=1,2,3,...,m;\text{ }%
t=1,2,....,n\right\}  $ be a sample from the spatio temporal random process
$\left\{  Y_{t}(\mathbf{s}_{i})\right\}  $. Here we consider the estimation of
FV under the assumption that the process is intrinsically stationary both
spatially and temporally. We assume that the process $\left\{  Y_{t}%
(\mathbf{s}_{i})\right\}  $ observed is not corrupted by noise.

Consider the Frequency Variogram $G_{\mathbf{s,s+h}}^{x}(\omega
)=E|J_{\mathbf{s}}^{y}(\omega)-J_{\mathbf{s}+\mathbf{h}}^{y}(\omega
)|^{2},|\omega|\leq\pi$. We noted earlier that the FV $G_{\mathbf{s,s+h}}%
^{x}(\omega)$ is \ the expected value of the periodogram of the incremental
process $X_{t}^{\mathbf{h}}(\mathbf{s})=Y_{t}(\mathbf{s})-Y_{t}(\mathbf{s}%
+\mathbf{h}),(t=1,2,...)$. The process$\{X_{t}^{\mathbf{h}}(\mathbf{s})\}$ is
spatially, temporally stationary when $\mathbf{h}$ is fixed. Therefore for
large $n$, it is well known that the periodogram is an unbiased estimator of
the second order spectral density function of the stationary process
$\{X_{t}^{\mathbf{h}}(\mathbf{s})\}$ though it is not a consistent estimator.
Therefore, our object here is to obtain a consistent estimator of the spectrum
of the incremental process $\{X_{t}^{\mathbf{h}}(\mathbf{s})\}$ for a given
$\mathbf{h}$, using the entire sample of discrete of Fourier transforms
$\left\{  J_{\mathbf{s}_{i}}(\omega_{k});i=1,2,...,m\right\}  $, for all
$\omega_{k}=\frac{2\pi k}{n},(k=0,1,....,\left[  \frac{n}{2}\right]  )$.

Let $\ g_{\mathbf{s}_{i},h}^{x}(\omega)$ be the second order spectrum of the
incremental process $\{X_{t}^{\mathbf{h}}(\mathbf{s}_{i})\}$. Since the
intrinsic process is spatially stationary $g_{\mathbf{s}_{i},\mathbf{h}}%
^{x}(\omega)$ does not depend $\mathbf{s}_{i}$. We denote such a stationary
spectrum of the intrinsic process by $g_{\mathbf{h}}^{x}(\omega)$.

Let $\Omega$ denote the set of all location $\mathbf{s}{_{1},\mathbf{s}%
_{2},...\mathbf{s}_{m}}$, and let $N(\mathbf{h})$ denote the subset of
locations, such that $N(\mathbf{h})=\{\mathbf{s}_{i};i=1,2,...,m$, $\text{such
that}$, \textbf{ }both $\mathbf{s}_{i},\mathbf{s}_{i}+\mathbf{h}\in\Omega\}$,.
$|N(\mathbf{h})|$ be the number of distinct elements in the set $N(\mathbf{h}%
)$. The estimation of stationary spectrum of a time series is well known and,
therefore, we discuss the estimation of $g_{\mathbf{h}}^{x}(\omega)$ only
briefly. For details, we refer to \cite{Priest}, \cite{Brilli}, \cite{Brock}.

Consider the estimator,
\begin{equation}
\hat{g}_{\mathbf{h}}^{x}(\omega)=\int_{-\pi}^{\pi}W_{n}(\omega-\theta)\left(
\frac{1}{|N(\mathbf{h})|}\sum_{i}I_{\mathbf{s}_{i},\mathbf{s}_{i}+\mathbf{h}%
}^{x}(\theta)d\theta\right)  , \label{estimator23}%
\end{equation}
where the sum is taken over the set $N(\mathbf{h})$, and the weight function
$W_{n}(\theta)$, which is a real valued even function of $\theta$, satisfies
the following Assumptions. For further details, see (\cite{Priest},
\cite{Brilli}).

\textbf{Assumptions:}

\begin{enumerate}
\item $W_{n}(\theta)\geq0$ for all $n$ and $\theta$,

\item $\int W_{n}(\theta)d\theta=1$, all $n$,

\item $\int W_{n}^{2}(\theta)d\theta<\infty$, all $n$,

\item For any $\varepsilon\left(  >0\right)  $, $W_{n}\left(  \theta\right)
\rightarrow0$, uniformly as $n\rightarrow\infty$, for $\left\vert
\theta\right\vert >\varepsilon$.
\end{enumerate}

\begin{theorem}
Let $g_{\mathbf{h}}^{x}(\omega)$ be the spectral density function of the
process $\{X_{t}^{\mathbf{h}}(\mathbf{s}_{i})\}$ for all $\mathbf{s}_{i}$ and
let $g_{\mathbf{s}_{i},\mathbf{s}_{j}}^{x}(\mathbf{h},\omega)$ be the cross
spectral density function of the process $\{X_{t}^{\mathbf{h}}(\mathbf{s}%
_{i})\}$ and $\{X_{t}^{\mathbf{h}}(\mathbf{s}_{j})\}$. Then we have
\begin{equation}
E(\hat{g}_{\mathbf{h}}^{x}(\omega))=g_{\mathbf{h}}^{x}(\omega)+O(\frac{\ln
n}{n}), \label{Equ_Th2_1}%
\end{equation}
and
\begin{equation}
\lim_{n\rightarrow\infty}Var(\hat{g}_{h}(\omega))=\frac{1}{|N(\mathbf{h}%
)|^{2}}\frac{2\pi}{n}\int{W_{n}^{2}}(\omega-\theta)\left[  \sum_{i,j}%
|g_{\mathbf{s}_{i},\mathbf{s}_{j}}^{x}(\mathbf{h},\theta)|^{2}\right]
d\theta. \label{Equ_Th2_2}%
\end{equation}

\end{theorem}

\begin{proof}
Take expectations both sides of (\ref{estimator23}),%
\[
E\left(  \hat{g}_{h}^{x}(\omega)\right)  =\int W_{n}(\omega-\theta)(\frac
{1}{|N(\mathbf{h})|})\sum_{i}E(I_{\mathbf{s}_{i},\mathbf{s}_{i}+\mathbf{h}%
}^{x}(\theta))d\theta,
\]
and we have
\[
E(I_{\mathbf{s}_{i},\mathbf{s}_{i}+\mathbf{h}}^{x}(\theta))=g_{\mathbf{h}}%
^{x}(\theta)+O\left(  \frac{\ln n}{n}\right)  ,
\]
and, \text{therefore, we obtain}%
\[
E\left(  \hat{g}_{\mathbf{h}}^{x}(\omega)\right)  =g_{\mathbf{h}}^{x}%
(\omega)+O\left(  \frac{\ln n}{n}\right)  ,
\]
in view of the Assumption 2, and the fact that $W_{n}(\theta)$ is approaching
the Dirac-Delta function concentrating its mass at $\theta=0$. Therefore,
$\hat{g}_{\mathbf{h}}^{x}(\omega)$ is asymptotically an unbiased estimator of
$g_{\mathbf{h}}^{x}(\omega)$. As we have noted earlier, estimating the
frequency variogram is equivalent to (for large $n$) estimating the spectral
density $g_{\mathbf{h}}^{x}(\omega)$ of the intrinsic process $\{X_{t}%
^{\mathbf{h}}(\mathbf{s}_{i})\}$. To obtain an expression for the variance, we
consider a discrete approximation of $\widehat{g^{x}}_{\mathbf{h}}(\omega)$.
Our derivation here is heuristic, and to obtain an expression for the
covariance we assume the intrinsic process is Gaussian, even though this
assumption is not essential for proving normality or consistency (see
\cite{Brilli}, and \cite{giraitis2012large}). Consider the discrete
approximation of (8), and take variance both sides, we get
\begin{align*}
Var(\hat{g}_{\mathbf{h}}^{x}(\omega)=\frac{1}{{|N(\mathbf{h})|}^{2}}{\left(
\frac{2\pi}{n}\right)  }^{2}\sum_{P}  &  \sum_{P^{\prime}}W_{n}(\omega
-\theta_{P})W_{n}(\omega-\theta_{P^{\prime}})\\
&  \times Cov\left(  \sum_{i}I_{\mathbf{s}_{i},\mathbf{s}_{i}+\mathbf{h}}%
^{x}(\theta_{P}),\sum_{j}I_{\mathbf{s}_{j},\mathbf{s}_{j+\mathbf{h}}}%
^{x}(\theta_{P^{\prime}})\right)  ,
\end{align*}
and we have%
\begin{align*}
&  Cov\left(  \sum_{i}I_{\mathbf{s}_{i},\mathbf{s}_{i}+\mathbf{h}}^{x}%
(\theta_{P}),\sum_{j}I_{\mathbf{s}_{j},\mathbf{s}_{j}+\mathbf{h}}^{x}%
(\theta_{P^{\prime}})\right) \\
&  =\eta(\theta_{p}-\theta_{p^{\prime}})\sum_{i}\sum_{j}{|g_{\mathbf{s}%
_{i},\mathbf{s}_{j}}^{x}(\mathbf{h},\theta_{p})|}^{2}+\eta(\theta_{p}%
+\theta_{p^{\prime}})\sum_{i}\sum_{j}{|g_{\mathbf{s}_{i},\mathbf{s}_{j}}%
^{x}(\mathbf{h},\theta_{p})|}^{2},
\end{align*}
where $\eta(\theta)=\sum_{-\infty}^{\infty}\delta(\theta-2\pi j)$ is a Dirac
comb (\cite{Brilli} Corollary 7.22). To obtain the above expression we used
the results already well known concerning the covariance between two
periodogram ordinates (see \cite{Brilli}).After substitution of this
expression for the covariance and after some simplification, we obtain
\[
\lim_{n\rightarrow\infty}Var(\hat{g}_{\mathbf{h}}(\omega))=\frac
{1}{|N(\mathbf{h})|^{2}}\frac{2\pi}{n}\int W_{n}^{2}(\omega-\theta)\left[
\sum\sum(g_{\mathbf{s}_{i},\mathbf{s}_{j}}^{x}(\mathbf{h},\theta
))^{2})\right]  d\theta.
\]

\end{proof}

The above \ result shows that $\hat{g}_{h}^{x}(\omega)$ is a mean square
consistent estimator of $g_{h}^{x}(\omega)$ \ and as we mentioned earlier that
$g_{h}^{x}(\omega)$ \ is asymptotically equivalent to the Frequency Variogram.

\begin{remark}
In the derivation of the above results, we have only assumed that the
intrinsic process is Gaussian. The assumption of Gaussianity is made only to
obtain a simple expression for the variance. The result that the estimator
$\widehat{g}_{\mathbf{h}}^{x}(\omega)$ is a consistent estimator is still
valid under non Gaussianity assumption (see \cite{Brilli}).
\end{remark}

\begin{remark}
It is well known that the usual Matheron estimator for the Variogram
$\gamma(\mathbf{h,u)}$ may not be stable if the data are sparse or irregularly
shaped (see \cite{Schabenberger2005},p 153). In such situations, it is usual
to consider all pairs $(\mathbf{s}_{i},\mathbf{s}_{j})$ such that
$\ \mathbf{s}_{i}-\mathbf{s}_{j}=\mathbf{h\pm\bigtriangleup,}$where
$\mathbf{\bigtriangleup}$ is tolerance (see \cite{Hucre}). The choice of
$\mathbf{\bigtriangleup}$ is arbitrary and the derivation of the sampling
properties become complicated.
\end{remark}

We can show by following similar lines as above, that as $n\rightarrow\infty$,
(see \cite{Priest})%
\[
Cov\left(  \hat{g}_{\mathbf{h}}^{x}(\omega_{1}),\hat{g}_{\mathbf{h}}%
^{x}(\omega_{2})\right)  =0\quad\text{for}\quad\omega_{1}+\omega_{2}\neq0.
\]
The asymptotic normality of $\hat{g}_{\mathbf{h}}^{x}(\omega)$ can be shown
using the results of \cite{hannan}, \cite{Tani}, \cite{Deo}.

\section{Complex Stochastic Partial Differential Equation(CSPDE) for the
intrinsic process and the spectrum for the FV.}

In a recent paper, \cite{2013arXivSRT_GyT} defined a complex stochastic
partial differential equation for the spatio temporal process and obtained an
analytic expression for the spectrum of the spatio-temporal process. The
parametric spectrum thus obtained from the assumed model is non-separable.A
spatio-temporal random process is said to be separable if its second order
space-time spectrum can be written as a product of two positive semi-definte
functions which are ,in fact , \ space spectrum which is a function of wave
numbers $\mathbf{\lambda}$ , and the other part corresponds to temporal
spectrum corresponding to the temporal frequency $\omega$. As we mentioned
earlier, stationarity assumption may not be realistic always, and therefore, a
weaker assumption that the process is intrinsically stationary is made. Here
our object is to define a model for such an intrinsic process, and obtain an
analytic parametric expression for the spectrum for the intrinsic process. In
a later section, we consider the estimation of the parameters of the spectral
function. We may note that \cite{Yag} and \cite{Huang} and others have
obtained spectra for the variogram in the case of spatial process. \cite{Yun},
\cite{Huang}, have considered non-parametric estimation of the variogram.

Consider the incremental random process ${X_{t}^{\mathbf{h}}}(s)=Y_{t}%
(\mathbf{s})-Y_{t}(\mathbf{s}+\mathbf{h})$, $\mathbf{s}\in\mathbb{R}^{d}%
,t\in\mathbb{Z}$. For a fixed $\mathbf{h}$, the incremental process is a
function of the spatial location $\mathbf{s}\in\mathbb{R}^{d}$, and time
$t\in\mathbb{Z}$

We consider the process $\{{X_{t}^{\mathbf{h}}}(s)\}$ which is assumed to be a
zero mean, stationary process in space and time. Define the DFT of the time
series $\{{X_{t}^{\mathbf{h}}}(s)\}$,
\[
J_{\mathbf{s,s+h}}^{(x)}(\omega_{k})=\frac{1}{\sqrt{2\pi n}}\sum_{t=1}%
^{n}{X_{t}^{\mathbf{h}}}(s)e^{it\omega_{k}},
\]
and the \ DFT $J_{\mathbf{s(L),s(L)+h}}^{(x)}(\omega_{k})$ of the time
series$\{{X_{t}^{\mathbf{h}}}(\mathbf{s(L)})\}$ where, for each $t,{X_{t}%
^{\mathbf{h}}}\mathbf{(s(L))}=Y_{t}(\mathbf{s+L})-Y_{t}(\mathbf{s}%
+\mathbf{L+h})$ at the frequencies
\[
\omega_{k}=\frac{2\pi k}{n},\text{ }(k=0,1,2,...,\left[  \frac{n}{2}\right]
).
\]

Define the covariance between two distinct Fourier Transforms
$J_{\mathbf{s,s+h}}^{(x)}(\omega)$ and $\ J_{\mathbf{s(L),s(L)+h}}%
^{(x)}(\omega),$
\[
g_{\mathbf{s,s+L}}^{(\mathbf{h})}(\omega)=Cov(J_{\mathbf{s,s+h}}^{(x)}%
(\omega)\text{,}\ J_{\mathbf{s(L),s(L)+h}}^{(x)}(\omega)),
\]
where $J_{\mathbf{s,s+h}}^{(x)}(\omega)$,$\ J_{\mathbf{s(L),s(L)+h}}%
^{(x)}(\omega)$, respectively, are discrete Fourier Transforms of the
incremental processes
\[
{X_{t}^{\mathbf{h}}}(s)=Y_{t}(\mathbf{s})-Y_{t}(\mathbf{s+h}),\quad
\text{and}\quad{X_{t}^{\mathbf{h}}}(\mathbf{s}(\mathbf{L}))=Y_{t}%
(\mathbf{s+L})-Y_{t}(\mathbf{s+L+h}),
\]
for $t=1,...,n,\quad\mathbf{s}=\mathbf{s}_{1},...,\mathbf{s}_{m}$ and
$\mathbf{L}\in\mathbf{R}^{d}$. We note that in computing the above, we fix
$\mathbf{h}$ and consider$\{{X_{t}^{\mathbf{h}}}(\mathbf{s})\}$ as one
spatio-temporal series.\newline Since the process $\{{X_{t}^{\mathbf{h}}%
}(\mathbf{s})\}$ is a zero mean second order spatially, temporally stationary,
it has the spectral representation.
\[
{X_{t}^{\mathbf{h}}}(\mathbf{s})=\int_{\mathbf{R}^{d}}\int_{-\pi}^{\pi
}e^{i(\mathbf{s}\cdot\boldsymbol{\lambda}+t\omega)}d{\xi_{X}^{(\mathbf{h})}%
}(\boldsymbol{\lambda},\omega).
\]
where $d{\xi_{X}^{(\mathbf{h})}}(\boldsymbol{\lambda},\omega)$ is a zero mean
complex random process with orthogonal increments with%
\begin{align*}
E[d{\xi_{X}^{(\mathbf{h})}}(\boldsymbol{\lambda},\omega)]  &  =0,\\
E|d{\xi_{X}^{(\mathbf{h})}}(\boldsymbol{\lambda},\omega)|^{2}  &  =d{F_{X}%
}^{(\mathbf{h})}(\boldsymbol{\lambda},\omega)={f_{X}^{\left(  \ \mathbf{h}%
\right)  \ }}(\boldsymbol{\lambda},\omega)d\boldsymbol{\lambda}d\omega.
\end{align*}
We define ${f_{X}}^{(\mathbf{h})}(\boldsymbol{\lambda},\omega),$ as the
spectral density function of the stationary intrinsic process $\{{X_{t}%
^{\mathbf{h}}}(\mathbf{s})\}.$ We have the following spectral representation
for the DFT of the intrinsic process.

\begin{proposition}
Let $J_{\mathbf{s,s+h}}^{(x)}(\omega)$ be the DFT of the stationary time
series $\{{X_{t}^{\mathbf{h}}}(\mathbf{s})\}$. Then,%

\[
J_{\mathbf{s,s+h}}^{(x)}(\omega)=\sqrt{\frac{n}{2\pi}}\int e^{i\mathbf{s\cdot
}\boldsymbol{\lambda}}d{\xi_{X}}^{(\mathbf{h})}(\boldsymbol{\lambda}%
,\omega)+o_{p}(1).
\]

\end{proposition}

\begin{proof}
The proof is similar to the proof given in Proposition 2 of
\cite{2013arXivSRT_GyT} and hence the details are omitted.
\end{proof}

In the following we denote the $d$ coordinates of the location $\mathbf{s}$ by
$(s_{1},s_{2}...s_{d}).$

\begin{theorem}
Let $\{J_{\mathbf{s}_{i}\mathbf{,s}_{i}\mathbf{+h}}^{(x)}(\omega
);i=1,2,...,m\}$ be the discrete Fourier transforms of the incremental process
$\{{X_{t}^{\mathbf{h}}}(\mathbf{s}_{i})\}$. Let
\begin{equation}
\left[  \sum_{i=1}^{d}\frac{\partial^{2}}{\partial s{_{i}}^{2}}-|P_{\mathbf{h}%
}(\omega,\boldsymbol{\psi})|^{2}\right]  ^{\nu}J_{\mathbf{s,s+h}}^{(x)}%
(\omega)={J}_{\eta_{\mathbf{s}}}^{(\mathbf{h})}(\omega),\quad|\omega|\leq\pi,
\label{A2}%
\end{equation}
where $\nu>0$,and ${J}_{\eta_{\mathbf{s}}}^{(\mathbf{h})}(\omega)$ is the DFT
of the space-time white noise process $\left\{  {\eta_{t}}(\mathbf{s}%
)\right\}  $ and $P_{\mathbf{h}}(\omega,\boldsymbol{\psi})$ is a polynomial
\ in $\omega$ and it is a function of some parameter vector $\boldsymbol{\psi
}$. Then the second order space-time spectrum \ of the intrinsic process
$\{{X_{t}^{\mathbf{h}}}(\mathbf{s})\}$ is given by
\begin{equation}
{f_{X}^{(\mathbf{h})}}(\boldsymbol{\lambda},\omega)=\frac{{\sigma_{\eta}}^{2}%
}{(2\pi)^{d+1}}\frac{1}{\left(  \sum_{i=1}^{d}{\lambda}_{i}^{2}+|P_{\mathbf{h}%
}(\omega,\boldsymbol{\psi})|^{2}\right)  ^{2\nu}}, \label{A3}%
\end{equation}
and the covariance between the periodograms (which is a spectrum dependent on
spatial distance $\mathbf{L}$, and the temporal frequency $\omega$) is given
by
\begin{align}
{g}_{\mathbf{s,s+L}}^{(\mathbf{h})}(\omega)  &  =Cov(J_{\mathbf{s,s+h}}%
^{(x)}(\omega),\ J_{\mathbf{s(L),s(L)+h}}^{(x)}(\omega))\nonumber\\
&  =\frac{{\sigma_{\eta}}^{2}}{(2\pi)^{d}2^{2\nu-1}\Gamma(2\nu)}\left(
\frac{||\mathbf{L}||}{|P_{\mathbf{h}}(\omega,\boldsymbol{\psi})|}\right)
^{2\nu-\frac{d}{2}}K_{2\nu-\frac{d}{2}}\left(  \left\Vert \mathbf{L}%
\right\Vert |P_{\mathbf{h}}(\omega,\boldsymbol{\psi})|\right)  . \label{A4}%
\end{align}
where $\mathbf{s(L)=s+L,}$and $K_{\nu}(x)$ is the modified Bessel function of
the second kind of order $\nu$. We note that in view of spatial stationarity,
the right hand side expression does not depend on $\mathbf{s}$, and depends
only on the Euclidean spatial distances $||\mathbf{L}||$ and $||\mathbf{h}||$.
Further, as $||\mathbf{L}||\rightarrow0$, the \ temporal spectrum of the
intrinsic process$\{{X_{t}^{\mathbf{h}}}(s)\}$ is given by
\begin{equation}
{g_{0}^{(\mathbf{h})}}(\omega)=Var(J_{\mathbf{s,s+h}}^{(x)}(\omega
))=\frac{{\sigma_{\eta}}^{2}}{(2\pi)^{\frac{d}{2}}2^{\frac{d}{2}}\left(
|P_{\mathbf{h}}(\omega,\boldsymbol{\psi})|^{2}\right)  ^{2v-\frac{d}{2}}}%
\frac{\Gamma(2v-\frac{d}{2})}{\Gamma\left(  2v\right)  }. \label{A5}%
\end{equation}

\end{theorem}

\begin{proof}
The proof is similar to the proof of Theorem 1 of the paper by
\cite{2013arXivSRT_GyT} and hence omitted.
\end{proof}

We note \ from the expression (12) for the space-time spectrum corresponding
to the process satisfying the model (11) , that it corresponds to a
non-separable process ,defined earlier. We also note further that as pointed
out by one reviewer, that the assumption that the random process$\ \left\{
{\eta_{t}}(\mathbf{s})\right\}  $ \ is a white noise process in spatial
coordinate $\mathbf{s}$ is a fiction , but still this assumption made in the
literature. More over both covariance function, and the variance given above
depend on $\mathbf{h}$ since the polynomial $P_{\mathbf{h}}(\omega
,\boldsymbol{\psi})$ is related to the second order spectral density function
of the intrinsic process ${X_{t}^{\mathbf{h}}}(\mathbf{s})$. We note that
${g_{0}^{(\mathbf{h})}}(\omega)$ depends on some parameters, say, $\psi.$We
denote this function by ${g_{0}^{(\mathbf{h})}}(\omega,\mathbf{\psi}).$

\begin{proposition}
Let $d=2,\quad\nu=1$ and assume $\mathbf{h}$ is fixed. Then
\begin{equation}
{g_{0}^{(\mathbf{h})}}(\omega,\mathbf{\psi})=\frac{{\sigma_{\eta}}^{2}}{4\pi
}|P_{\mathbf{h}}(\omega,\boldsymbol{\psi})|^{-2}. \label{A6}%
\end{equation}

\end{proposition}

The above result shows that the function $|P_{\mathbf{h}}(\omega
,\boldsymbol{\psi})|^{2}$ is related to the stationary temporal spectrum of
the process $\left\{  {X_{t}^{\mathbf{h}}}(\mathbf{s})\right\}  $. We note
further that $f_{X}^{h}(\boldsymbol{\lambda},\omega)$ is the spatio-temporal
spectrum and ${g_{0}^{(\mathbf{h})}}(\omega,\mathbf{\psi})$ is the stationary
temporal spectrum of the process $\left\{  {X_{t}^{\mathbf{h}}}(\mathbf{s}%
)\right\}  $. For large $n$ and for a fixed $\mathbf{h}$,
$Var(J_{\mathbf{s,s+h}}^{(x)}(\omega))\approx{g_{0}^{(\mathbf{h})}}%
(\omega,\mathbf{\psi}),\quad|\omega|\leq\pi$. Once again we note that the
spectral density function ${f_{X}^{(\mathbf{h})}}(\boldsymbol{\lambda}%
,\omega)$ is non-separable. \newline In the above, we have shown that we can
obtain a parametric expression, in a close form, for the spectral density
function of the intrinsic process. The spectral density function is given by
${g_{0}^{(\mathbf{h})}}(\omega,\mathbf{\psi}).$ In the following section, we
consider the estimation of parameter vector $\underline{\boldsymbol{\psi}}$
using the discrete Fourier transform of the process$\{$ ${X_{t}^{\mathbf{h}}%
}(\mathbf{s})\}$.

\section{Estimation of the parameters of the frequency variogram of the
intrinsic process.}

\cite{math}, \cite{Cressie}, \cite{Stein}, \cite{Yun}, and many others have
stressed the importance of the variogram in Kriging and in view of this,
several methods of estimation of the variogram in the case of spatial
processes have been proposed. \cite{Yun} have proposed nonparametric
estimation of the variogram, \cite{Huang} proposed the estimation of the
variogram and its spectrum. If one assumes that the intrinsic process
satisfies a specific model \ which have parameters which are usually unknown,
then one needs to estimate the parameters of the model. In section 7,we have
obtained an expression for \ the spectral density function \ of the intrinsic
process assuming that the process satisfies the model given in Theorem 2 and
we have seen that the spectrum depends on the parameter vector $\psi$. We
consider the estimation of the parameter vector $\psi$ from the data $\left\{
{X_{t}^{\mathbf{h}}}(\mathbf{s})\right\}  $.

Our object here is to estimate $\boldsymbol{\psi}$ of ${g_{0}^{(\mathbf{h})}%
}(\omega,\boldsymbol{\psi})$ given the discrete Fourier Transforms
$\{J_{\mathbf{s,s+h}}^{(x)}(\omega_{k});i=1,2...m;k=1,2...[\frac{n}{2}]$
\ obtained from the intrinsic processes $\{{X_{t}^{\mathbf{h}}}(\mathbf{s}%
_{i});t=1,2,...,n;i=1,2,...m\}$. Let the set $N(\mathbf{h})=\{\mathbf{s}%
_{i};i=1,2,...,m$, \textbf{ }$\mathbf{s}_{i},\mathbf{s}_{i}+\mathbf{h}%
\in\Omega\}$. If we \ are assuming that the DFT of the intrinsic process
satisfies the model (11) stated in the Theorem 2, then the parameters we have
to consider for the estimation are $\boldsymbol{\psi}$ of the Polynomial
$P_{\mathbf{h}}(\omega,\boldsymbol{\psi})$ related to the temporal spectrum
${g_{0}}^{(\mathbf{h})}(\omega,\mathbf{\psi})$ of the process\{ ${X_{t}%
^{\mathbf{h}}}(\mathbf{s}_{i})\}$. Here we obtain the likelihood function
using the DFT's, and the approach is similar to the method described in
\cite{Subba1}. We refer to the above paper for details.

Consider the Discrete Fourier Transforms $\{J_{\mathbf{s}_{i}\mathbf{,s}%
_{i}\mathbf{+h}}^{(x)}(\omega_{k}))\}$ corresponding to the time series
$\{Y_{t}(\mathbf{s}_{i})\}$, $\{Y_{t}(\mathbf{s}_{i}+\mathbf{h})\}$. We note
that for large $n$, the \ complex valued random variable $J_{\mathbf{s}%
_{i}\mathbf{,s}_{i}\mathbf{+h}}^{(x)}(\omega_{k})$ is asymptotically
distributed as complex normal with mean zero and variance ${g_{0}%
}^{(\mathbf{h})}(\omega_{k},\boldsymbol{\psi})$ (see \cite{Brilli} and
\cite{giraitis2012large}) and independent over distinct frequencies.. Let
$M=\left[  \frac{n}{2}\right]  $. Consider the $M$ dimensional complex valued
random vector
\[
L_{\left\Vert \mathbf{h}\right\Vert }(\omega)=\{J_{\mathbf{s}_{i}%
\mathbf{,s}_{i}\mathbf{+h}}^{(x)}(\omega_{1}),J_{\mathbf{s}_{i}\mathbf{,s}%
_{i}\mathbf{+h}}^{(x)}(\omega_{2}),...,J_{\mathbf{s}_{i}\mathbf{,s}%
_{i}\mathbf{+h}}^{(x)}(\omega_{M})\},
\]
which is distributed asymptotically as complex multivariate normal with mean
zero and variance covariance matrix with diagonal elements
\[
\left[  {g_{0}^{\left\Vert \mathbf{h}\right\Vert }}(\omega_{1}%
,\boldsymbol{\psi}),{g_{0}^{\left\Vert \mathbf{h}\right\Vert }}(\omega
_{2},\boldsymbol{\psi}),...,{g_{0}}^{\left\Vert \mathbf{h}\right\Vert }%
(\omega_{M},\boldsymbol{\psi})\right]  .
\]

We note that off diagonal elements of the covariance matrix are zero.
Proceeding as in \cite{Subba1}, we can show that the log likelihood function
$l(\psi/J_{\mathbf{s,s+h}}(\omega))$ is proportional to
\[
{Q_{n,i}^{(\mathbf{h})}}(\boldsymbol{\psi})=\sum_{k=1}^{M}\left[  \ln{g_{0}%
}^{\left\Vert \mathbf{h}\right\Vert }(\omega_{k},\boldsymbol{\psi}%
)+\frac{{I_{\mathbf{s}_{i},\mathbf{s}_{i+\mathbf{h}}}^{\text{ }x}}(\omega
_{k})}{{g_{0}}^{(\mathbf{h})}(\omega_{k},\psi)}\right]  .
\]

Now consider all the locations $(\mathbf{s}_{i},\mathbf{s}_{i}+\mathbf{h})$;
$i=1,2,...,m$ belonging to the set $N(\mathbf{h})$. Then we have the pooled
criterion%
\begin{equation}
Q_{n,N(\mathbf{h})}(\boldsymbol{\psi})=\frac{1}{|N(\mathbf{h})|}%
\sum_{(\mathbf{s}_{i},\mathbf{s}_{i}\in N(\mathbf{h}))}{Q_{n,i}^{(\mathbf{h}%
)}}(\boldsymbol{\psi}). \label{poobd1}%
\end{equation}

Suppose we have \ $H$ spatial distances $\{\mathbf{h}(l);l=1,2,...H\mathbf{\}}%
$ for which the intrinsic stationarity condition is satisfied then we can
define an over all measure for minimization,
\begin{equation}
Q_{n}(\boldsymbol{\psi})=\frac{1}{H}\sum Q_{n,N(\mathbf{h}_{l})}%
(\boldsymbol{\psi}). \label{A7}%
\end{equation}
We minimize \ (17) with respect to $\boldsymbol{\psi}$. The asymptotic
normality of the estimator of $\boldsymbol{\psi}$ can be proved using the
methodology described in \cite{Subba1}. For large $n$, we can show
\[
\sqrt{n}(\tilde{\boldsymbol{\psi}}-\boldsymbol{\psi})\overset{D}{\rightarrow
}N(0,\left[  \nabla^{2}Q_{n}(\boldsymbol{\psi})\right]  ^{-1}V\text{ }\left[
\nabla^{2}Q_{n}(\boldsymbol{\psi})\right]  \text{ }),
\]
where $V=\lim_{n\rightarrow\infty}Var\left[  \frac{1}{\sqrt{n}}\nabla
Q_{n,}(\boldsymbol{\psi})\right]  $, and $\nabla Q_{n}(\boldsymbol{\psi})$ is
a Jacobian vector of first order partial derivatives, and $\left[  \nabla
^{2}Q_{n}(\boldsymbol{\psi})\right]  $ is a Hessian matrix of second order
partial derivatives.

\section{Test for independence of $m$ spatial time series.}

So far we have considered the analysis of spatio-temporal data using various
frequency domain methods. We assumed that there is a second order dependence
in space and time. It is important to test for Independence over space and
time before modelling the data. \cite{Hine} proposed a test statistic for
testing spatio temporal independence; and the test proposed is as an extension
of Moran's test. In their book \cite{Cressie2} briefly discussed the test. In
this section, we propose a test for spatial independence using the Discrete
Fourier Transforms and the test is based on the test proposed by \cite{Wahb}
which is an extension of the classical test for independence used in
multivariate analysis. Here we briefly describe the test. Let.
\[
\underline{Y}_{t}^{\prime}=(Y_{t}(\mathbf{s}_{1}),Y_{t}(\mathbf{s}%
_{2}),...,Y_{t}(\mathbf{s}_{m})).
\]

We say, the multivariate time series $\{\underline{Y}_{t}\}$ is second order
stationary if (see \cite{Brock})

\begin{enumerate}
\item $E(\underline{Y}_{t})=\underline{\mu},$

\item $E(\underline{Y}_{t}-\underline{\mu})(\underline{Y}_{t+p}-\underline
{\mu})^{^{\prime}}=\underline{\Gamma}(p),$ where%
\begin{align*}
\underline{\mu}^{^{\prime}}  &  =(\mu_{1},\mu_{2},...,\mu_{m}),\\
\Gamma(p)  &  =(\sigma_{ij}(p)),\\
\sigma_{ij}(p)  &  =E\left(  Y_{t}(\mathbf{s}_{i})-\mu_{i}\right)  \left(
{Y}_{t+p}(\mathbf{s}_{j})-\mu_{j}\right)  ,(i,j=1,2,...,m),\\
\sigma_{ij}(p)  &  =\sigma_{ji}(-p).
\end{align*}

\end{enumerate}

Here we are assuming that the spatio-temporal data is temporally stationary
only and no assumption of spatial stationarity is assumed. We assume further
that $\underline{Y}_{t}$ is Gaussian. Define the complex valued random vector
\[
\underline{J}^{\prime}(\omega_{k})=(J_{\mathbf{s}_{1}}^{y}(\omega
_{k}),J_{\mathbf{s}_{2}}^{y}(\omega_{k}),...,J_{\mathbf{s}_{m}}^{y}(\omega
_{k})),
\]
where $J_{\mathbf{s}_{i}}^{y}(\omega_{k})$ is the DFT of the time series data
$\{Y_{t}(\mathbf{s}_{i})\}$, and $\omega_{k}=\frac{2\pi k}{n}$,
$(k=,1,...,\left[  \frac{n}{2}\right]  )$. We know that the random vector
$\underline{J}(\omega_{k})$ is distributed as complex normal with mean
$\mathbf{0}$ and variance covariance matrix $\underline{F}(\omega_{k}),$ where
$\underline{F}(\omega_{k})=\left[  E(J_{\mathbf{s}_{i}}^{y}(\omega
_{k}){J_{\mathbf{s}_{j}}^{y\ast}}(\omega_{k})\right]  $. We note that
$\underline{F}(\omega_{k})$ is a Hermitian matrix, with elements
\[
f_{\mathbf{s}_{i},\mathbf{s}_{j}}(\omega_{k})=E(J_{\mathbf{s}_{i}}^{y}%
(\omega_{k})J_{\mathbf{s}_{j}}^{y\ast}(\omega_{k}))=f_{\mathbf{s}%
_{j},\mathbf{s}_{i}}(-\omega_{k}).
\]
In the above $f_{\mathbf{s}_{i},\mathbf{s}_{i}}(\omega_{k})$ is the second
order spectral density function of the process $\left\{  Y_{t}(\mathbf{s}%
_{i})\right\}  $, and $f_{\mathbf{s}_{i},\mathbf{s}_{j}}(\omega_{k})$ is the
cross spectral density function of the process $\left\{  Y_{t}(\mathbf{s}%
_{i})\right\}  $ and $\left\{  Y_{t}(\mathbf{s}_{j})\right\}  $. The cross
spectral density function is usually a complex valued function.

If we assume that the spatio-temporal process $\left\{  Y_{t}(\mathbf{s}%
)\right\}  $ is stationary in space and time, and further assume that the
process is isotropic in space, then%

\begin{align*}
f_{\mathbf{s}_{i},\mathbf{s}_{i}}(\omega)  &  =f_{0}(\omega),\\
f_{\mathbf{s}_{i},\mathbf{s}_{j}}(\omega)  &  =f_{||\mathbf{s}_{i}%
-\mathbf{s}_{j}||}(\omega).
\end{align*}
In this case the matrix $\underline{F}(\omega)$ is real and symmetric, and all
the diagonal elements are equal to $f_{0}(\omega)$.

As pointed out earlier, for testing spatial independence we do not need the
assumption of spatial stationarity. Below we assume that the process is
Gaussian. Under the null hypothesis that the spatial process is spatially
independent, the spectral matrix $F(\omega)$ is a diagonal matrix for all
$|\omega|\leq\pi$. For constructing the test, we proceed as in \cite{Wahb}.
Consider the discrete Fourier transforms defined earlier. For each location
$s_{i}$, let the Fourier transform be given by $(J_{s_{i}}^{y}(\omega_{l}))$
where $\omega_{l}=\frac{2\pi j_{l}}{n}$, $j_{l}=(l-1)(2k+1)+(k+1)$;
$l=1,2,...,M_{1}$ where $M_{1}$ is chosen such that $2(k+1)M_{1}=\frac{n-1}%
{2}$. (Here we assume that the number of observations n, is odd.) As in
\cite{Wahb} we define the cross spectral estimator of $f_{s_{i},\mathbf{s}%
_{j}}(\omega)$ by
\[
\hat{f}_{s_{i},\mathbf{s}_{j}}(\omega_{l})=\frac{1}{2k+1}\sum_{j_{1}=-k}%
^{k}I_{i,j}(\omega_{l}+\frac{2\pi j_{1}}{n}),\quad(l=1,2,...,M_{1}),
\]
where the cross periodogram $I_{ij}(\omega_{l})=J_{s_{i}}^{y}(\omega
_{l}){J_{\mathbf{s}_{j}}^{y\ast}}(\omega_{l})$.

Let $\hat{F}(\omega_{l})=(\hat{f}_{s_{i},\mathbf{s}_{j}}(\omega_{l}%
))\quad(l=1,2,...,M_{1})$.

We note that the random matrices $\hat{F}(\omega_{l});l=1,2,...,M_{1}$, for
large $k$, are approximately distributed as random matrices $\tilde{F}%
(\omega_{l}),(l=1,2,...,M)$ which are distributed as complex Wishart, usually
denoted by $W_{c}(F,m,2k+1)$. \cite{Wahb} has shown that the likelihood ratio
test for testing the null hypothesis that the matrices $F(w_{l})$ are diagonal
for all $\left\{  \omega_{l}\right\}  $ leads to the test statistic, for each
$w_{l}$,%
\[
\tilde{\boldsymbol{\lambda}}_{l}=\frac{|\tilde{F}(\omega_{l})|}{\prod
_{j=1}^{m}\tilde{f}_{\mathbf{s}_{j},\mathbf{s}_{j}}(\omega_{l})}\quad\left(
l=1,2,...,M_{1}\right)  ,
\]
and the over-all test statistic to consider is $\boldsymbol{\Lambda}=-\frac
{1}{M_{1}}\sum\ln\tilde{\boldsymbol{\lambda}_{l}}$. For large $k$ and $M_{1}$,
under the null hypothesis, the statistic $\boldsymbol{\Lambda}$ is
asymptotically distributed as normal with mean%
\[
E(\boldsymbol{\Lambda})=\sum_{j=1}^{m-1}\frac{m-j}{k^{\prime}-j}%
\]
and variance%
\[
Var(\boldsymbol{\Lambda})=\frac{1}{M_{1}}\sum_{j=1}^{m-1}\frac{m-j}{\left(
k^{\prime}-j\right)  ^{2}}%
\]
where $k^{\prime}=2k+1$. Under the null hypotheses of spatial independence,
for large $k$ and $M$, the statistic $S=\frac{\boldsymbol{\Lambda
}-E(\boldsymbol{\Lambda})}{\sqrt{Var(\boldsymbol{\Lambda})}}$ is distributed
as standard normal. We note that if for each $\mathbf{s}_{i}$, $\left\{
Y_{t}(\mathbf{s}_{i})\right\}  $ is a Gaussian white noise, then the spectral
density function is given by $f_{\mathbf{s}_{i},\mathbf{s}_{i}}(\omega
)=\frac{{\sigma_{s_{i}}}^{2}}{2\pi}$, where ${\sigma_{\mathbf{s}_{i}}}^{2}$ is
the variance of the white noise. If the null hypothesis is both spatially and
temporally independent then the diagonal elements of the matrix $F(\omega
_{l})$ will be proportional to $({\sigma_{\mathbf{s}_{1}}}^{2},{\sigma
_{\mathbf{s}_{2}}}^{2},{\sigma_{\mathbf{s}_{3}}}^{2},,...,{\sigma
_{\mathbf{s}_{m}}}^{2})$, and all off diagonal elements will be zero.

\section{Appendix: Discrete Fourier Transforms.}

\label{App:AppendixA} In this section, we will briefly summarize some results
related to the Discrete Fourier Transforms, further details, we refer to
\cite{2013arXivSRT_GyT}., \cite{Brilli}, \cite{giraitis2012large}.

Let $\left\{  Y_{t}(\mathbf{s})\right\}  $, where $\left\{  \mathbf{s}%
\in\mathbb{R}^{d};t\in Z\right\}  $ denote a zero mean second order
spatially,temporally stationary process with spectral representation
\begin{equation}
Y_{t}(\mathbf{s})=\int_{-\infty}^{\infty}\int_{-\pi}^{\pi}e^{i(\mathbf{s}%
\cdot\boldsymbol{\lambda}+t\omega)}dZ_{y}(\boldsymbol{\lambda},\omega),
\label{eq35}%
\end{equation}
and let $\left\{  Y_{t}(\mathbf{s_{i}}));i=1,2,...,m;t=1,2,...,n\right\}  $ be
a sample from the process$\left\{  Y_{t}(\mathbf{s})\right\}  $. We note that
$Z_{y}(\boldsymbol{\lambda},\omega)$ is a zero mean complex valued function
with orthogonal increments and%
\begin{align*}
E[dZ_{y}(\boldsymbol{\lambda},\omega)]  &  =0,\\
E|dZ_{y}(\boldsymbol{\lambda},\omega)|^{2}  &  =dF_{y}(\boldsymbol{\lambda
},\omega),
\end{align*}
where $dF_{y}(\boldsymbol{\lambda},\omega)$ is a spectral measure, Let
$dF_{y}(\boldsymbol{\lambda},\omega)=f_{y}(\boldsymbol{\lambda},\omega
)d\boldsymbol{\lambda}d\omega$, where $f_{y}(\boldsymbol{\lambda},\omega)$ is
the spatio-temporal spectral density function of the process $\left\{
Y_{t}(\mathbf{s})\right\}  $. Define the Discrete Fourier Transform%
\[
J_{\mathbf{s}}^{y}(\omega)=\frac{1}{\sqrt{2\pi n}}\sum_{t=1}^{n}%
Y_{t}(\mathbf{s})e^{it\omega},|\omega|\leq\pi,\text{
\ \ \ \ \ \ \ \ \ \ \ \ \ \ \ \ \ \ }(19)
\]

\begin{proposition}
Let the spectral representation of the process $\left\{  Y_{t}(\mathbf{s}%
_{i})\right\}  $ be given by \ref{eq35}, and let $J_{\mathbf{s}}(\omega)$ be
the DFT of the sample $\left\{  Y_{t}(\mathbf{s});t=1,2,...n\right\}  $. Then
we have

\begin{enumerate}
\item $Y_{t}(\mathbf{s})=\sqrt{\frac{n}{2\pi}}\int J_{\mathbf{s}}^{y}%
(\omega)e^{it\omega}d\omega,$ \label{eq1}

\item $J_{\mathbf{s}}^{y}(\omega)\approx\int e^{is\boldsymbol{\lambda}}%
\sqrt{\frac{n}{2\pi}}dZ_{y}(\boldsymbol{\lambda},\omega).$ \label{eq2}
\end{enumerate}
\end{proposition}

\begin{proof}
By substitution and using the properties of Dirac Delta function, one can show
(\ref{eq2}). (\ref{eq1}) follows by inversion of (19). For details, refer to
\cite{2013arXivSRT_GyT}.
\end{proof}

Let $I_{\mathbf{s}}^{y}(\omega_{k})=|J_{\mathbf{s}}^{y}(\omega_{k})|^{2}$ be
the periodogram. The following results are well known (\cite{Priest},
\cite{Brilli})

\begin{enumerate}
\item $E(I_{\mathbf{s}}^{y}(\omega_{k}))=g_{\mathbf{s}}^{y}(\omega
_{k})+O(n^{-1})$

\item $Var(I_{\mathbf{s}}^{y}(\omega_{k}))={g_{\mathbf{s}}^{y2}}(\omega
_{k})+O(n^{-1}),\quad\omega_{k}\neq0,\pi,$

\item $Cov(I_{\mathbf{s}}^{y}(\omega_{k}),I_{\mathbf{s}}^{y}(\omega
_{l})=O(n^{-1})\quad if\quad\omega_{k}+\omega_{l}\neq0\quad(mod\quad2\pi),$ In
view of spatial stationarity, $g_{\mathbf{s}}^{y}(\omega)=g_{0}^{y}(\omega)$
for all $\mathbf{s}$, and
\[
g_{\mathbf{s}}^{y}(\omega)=\frac{1}{2\pi}\sum_{k}Cov(Y_{t}(\mathbf{s}%
),Y_{t+k}\left(  \mathbf{s}\right)  )e^{-i\omega k},|\omega|\leq\pi
\]

\item $Cov(J_{\mathbf{s}_{i}}^{y}(\omega_{k}),J_{\mathbf{s}_{j}}^{y}%
(\omega_{k}))=O(n^{-1}),\quad if\quad\omega_{k}+\omega_{l}\neq0\quad
(mod\quad2\pi)$

\item $Cov(J_{\mathbf{s}_{i}}^{y}(\omega_{k}),J_{\mathbf{s}_{j}}^{y}%
(\omega_{k}))=\frac{1}{2\pi}\sum\limits_{n=-\infty}^{\infty}c\left(
\mathbf{s}_{i}-\mathbf{s}_{j},n\right)  e^{-in\omega_{k}}=g_{\mathbf{s}%
_{i}-\mathbf{s}_{j}}\left(  \omega_{k}\right)  +O(n^{-1}).$If the process is
isotropic then the spectral density function $g_{\mathbf{s}_{i}-\mathbf{s}%
_{j}}\left(  \omega_{k}\right)  =$ $g_{||\mathbf{s}_{i}-\mathbf{s}_{j}%
||}(\omega_{k})$ which is a real valued function. .
\end{enumerate}

\textbf{Acknowledgements}. Part of the research reported in this paper was
done when one of the authors (Subba Rao) was visiting the CRRAO AIMSCS,
University of Hyderabad Campus, India which was funded by a grant from the
Department of Science and Technology, Government of India (grant no.
SR/S4/516/07). Also,we would like to thank Professor Noel Cressie for bringing
to our attention the paper of \cite{Hine}, and also for \ his comments on an
earlier paper which lead to some results given in this paper. We would \ like
to thank Professor Liudas Giraitis, Queen Mary University, London and \ Dr
Suhasini Subba Rao, Texas A\&M University for reading the paper and for making
many helpful comments. We would like to thank the two reviewers for many
suggestions which improved the presentation.

\end{document}